\documentclass[a4paper]{article}
\usepackage{fullpage}
\usepackage{graphicx}

\usepackage{microtype}

\usepackage{amssymb}
\usepackage{amsmath}
\usepackage{amsthm}

\usepackage{xspace}
\usepackage{color}
\usepackage[shortlabels]{enumitem}

\newtheorem{theorem}{Theorem}
\newtheorem{lemma}[theorem]{Lemma}

\newcommand{\Reals}{\mathbb{R}}
\newcommand{\eps}{\varepsilon}
\newcommand{\Km}{K_\mu}
\newcommand{\surf}[1]{\mathrm{surf}(#1)}
\newcommand{\peri}[1]{\mathrm{peri}(#1)}
\newcommand{\vol}[1]{\mathrm{vol}(#1)}
\newcommand{\Ps}{P^{\mu}}
\newcommand{\Rs}{R^{\mu}}

\newcommand{\KK}{\mathfrak{K}}

\let\leq\leqslant
\let\geq\geqslant

\usepackage{hyperref}
\hypersetup{%
  breaklinks,%
  colorlinks=true,%
  linkcolor=[rgb]{0.45,0.0,0.0},%
  citecolor=[rgb]{0,0,0.45}
}

\makeatletter
\partopsep\z@ \textfloatsep 10pt plus 1pt minus 4pt
\def\section{\@startsection {section}{1}{\z@}%
  {-3.5ex plus -1ex
    minus -.2ex}{2.3ex plus .2ex}{\large\bf}}
\def\subsection{\@startsection{subsection}{2}%
  {\z@}{-3.25ex plus
    -1ex minus -.2ex}{1.5ex plus .2ex}{\normalsize\bf}}
\def\@fnsymbol#1{\ensuremath{\ifcase#1\or *\or 1\or 2\or 3\or 4\or
    5\or 6\or 7 \or 8\ or 9 \or 10\or 11 \else\@ctrerr\fi}}
\makeatother

\title{The Inverse Kakeya Problem%
  \thanks{This work was initiated at the 21st Korean Workshop on
    Computational Geometry, held in Rogla, Slovenia, in June 2018. We
    thank all workshop participants for their helpful comments.}}

\author{Sergio Cabello%
  \thanks{Faculty of Mathematics and Physics, University of Ljubljana, and IMFM, Slovenia}
  \and
  Otfried Cheong%
  \thanks{KAIST, Daejeon, Korea}
  \and
  Michael Gene Dobbins%
  \thanks{Department of Mathematical Sciences, Binghamton University,
  Binghamton, NY, USA}}

\begin{document}

\maketitle

\begin{abstract}
  We prove that the largest convex shape that can be placed inside a
  given convex shape~$Q \subset \Reals^{d}$ in any desired orientation
  is the largest inscribed ball of~$Q$.  The statement is true both
  when ``largest'' means ``largest volume'' and when it means
  ``largest surface area''.  The ball is the unique solution, except 
  when maximizing the perimeter in the two-dimensional case.
\end{abstract}

\section{Introduction}
\label{sec:intro}

The well-known Kakeya problem, originally asked by Soichi Kakeya in
1917, is the following question: What is the minimum area region~$Q$
in the plane in which a needle of length~$1$ can be turned through
$360^{\circ}$ continuously, and return to its initial
position~\cite{kakeya1917}?  When $Q$ is required to be convex, the
answer is the equilateral triangle of height one~\cite{pal1921}.  For
general~$Q$, however, Besicovitch showed that a region~$Q$ of measure
zero exists~\cite{besicovitch1920,besicovitch1928}.  Kakeya-type
problems have received considerable attention in the literature, as
there are strong connections to problems in number
theory~\cite{Bourgain2000}, geometric combinatorics~\cite{Wolff1999},
arithmetic combinatorics~\cite{Laba2008}, oscillatory integrals, and
the analysis of dispersive and wave equations~\cite{tao01}.

Being able to rotate a needle through~$360^{\circ}$ inside~$Q$ clearly
implies that it can be placed in~$Q$ in any desired orientation.  Bae
et al.~\cite{bae_et_al:LIPIcs:2018:8719} showed that the converse
holds more generally for convex shapes in the plane: If a planar
convex compact shape~$P$ can be placed in a planar convex compact
shape~$Q$ in any desired orientation, then~$P$ can also be rotated
through~$360^{\circ}$ inside~$Q$.  A natural generalization of
Kakeya's problem is therefore to ask, given a planar convex compact
shape~$P$, what is the minimum area convex shape~$Q$ such that~$P$ can
be placed in~$Q$ in any desired orientation.  This problem still seems
to be wide open, the answer is not even known when~$P$ is an
equilateral triangle or a square.

In this short note, we consider the inverse of this question: We are
given a convex compact shape~$Q \subset \Reals^{d}$, $d\geq 2$, 
and we ask: what is the largest shape~$P$ that can be 
placed in~$Q$ in any desired orientation?

We show that, independent of the shape of~$Q$, the answer is always a
spherical ball, and therefore~$P$ is the largest inscribed ball
of~$Q$.  The result is true both for maximizing the volume of~$P$ and
for maximizing the surface area of~$P$.  The answer is always unique,
\emph{except} when maximizing the perimeter of~$P$ in the planar case.
For instance, inside a unit square both a unit-diameter disk and a
unit-diameter Reauleaux-triangle can be turned.  Both have the same
perimeter, but the disk has larger area.

The result for maximizing the volume is a consequence of the
well-known Brunn-Minkowski theorem.  For the surface-area result, we
make use of the generalized Brunn-Minkowski theorem, a theorem that
deserves to be better known.  This proof does not cover the planar
case, so we give an elementary proof based on Minkowski sums.

Our characterization also solves the computational question of
computing the largest convex~$P$ that can be placed in any desired
orientation in a given convex polyhedron~$Q \subset \Reals^d$, since
the largest inscribed ball can be computed efficiently as a linear
program.

For completeness, let us observe that the case~$d=1$ is simple and has
the same behavior as~$d=2$. Indeed, in~$\Reals$ the convex shapes are
segments, there is a unique longest segment that can be placed inside
a segment~$Q$ in both orientations, namely~$Q$ itself, and all
segments of positive length inside~$Q$ have a boundary of the same
size, namely two points.

On the other hand, if we ask for the $P$ that maximizes the
\emph{diameter}, then the answer is different: it is a line segment
whose length is the smallest width of~$Q$.  In general, this is longer
than the diameter of the largest inscribed ball, for instance when
$Q$~is an equilateral triangle.

\section{Minkowski sums}

For two convex shapes $P$ and~$Q$ in~$\Reals^{d}$, the Minkowski sum
$P + Q$ is the set $\{p + q \mid p \in P, \, q \in Q\}$.  Minkowski
sums allow us to interpolate between two convex shapes~$P_0$
and~$P_1$: for~$0 \leq \lambda \leq 1$, we define~$P_{\lambda} :=
(1-\lambda) P_0 + \lambda P_1$.

\begin{lemma}
  \label{lem:fits}
  Let $Q$ be a convex shape in~$\Reals^{d}$, let $P_0, P_1 \subseteq
  Q$ be convex shapes, and let $0 \leq \lambda \leq 1$.
  Then~$P_\lambda \subseteq Q$.
\end{lemma}
\begin{proof}
  Let $p \in P_\lambda$.  Then $p = (1-\lambda) p_0 + \lambda p_1$,
  with $p_0 \in P_0 \subseteq Q$ and $p_1 \in P_1 \subseteq Q$.  Since
  $p_0, p_1 \in Q$ and $Q$ is convex, $p \in Q$.
\end{proof}

For a convex shape~$Q \subset \Reals^{d}$, we introduce~$\KK(Q)$ as
the family of all convex shapes~$P \subset \Reals^{d}$ that can be
placed in~$Q$ in any desired orientation.
Lemma~\ref{lem:fits} immediately implies the following:
\begin{lemma}
  \label{lem:freely}
  Let $Q$ be a convex shape in~$\Reals^{d}$, let $P_0, P_1 \in
  \KK(Q)$, and let $0 \leq \lambda \leq 1$.  Then~$P_\lambda \in
  \KK(Q)$.
\end{lemma}
\begin{proof}
  Consider an arbitrary rotation~$\rho$.  We need to argue that there
  exists a translation~$t$ such that $\rho P_\lambda + t \subset Q$.
  By assumption, there are translations $t_0$ and $t_1$ such that
  $\rho P_0 + t_0 \subset Q$ and $\rho P_1 + t_1 \subset Q$.  Setting
  $t = (1-\lambda) t_0 + \lambda t_1$, we have
  \[
  \rho P_\lambda + t
  = \rho \big((1-\lambda) P_0 + \lambda P_1 \big) + (1-\lambda) t_0 + \lambda t_1
  = (1-\lambda) (\rho P_0 + t_0) + \lambda (\rho P_1 + t_1).
  \]
  Now Lemma~\ref{lem:fits} implies the claim.
\end{proof}

We denote the $d$-dimensional volume of convex shape~$P$
by~$\psi_0(P)= \vol{P}$, and the $d-1$-dimensional volume of the
boundary of~$P$ by~$\psi_1(P)= \surf{P}$.  The key ingredient for our
proof is the following lemma:
\begin{lemma}
  \label{lem:halfway}
  Let $P_0, P_1 \subset \Reals^{d}$ be convex shapes with $\psi_w(P_0)
  = \psi_w(P_1)$, for $w \in \{0, 1\}$ and $d \geq 2 + w$.  Then
  $\psi_w(P_{1/2}) \geq \psi_w(P_0)$, and equality holds only when
  $P_0$ and $P_1$ are homothets.
\end{lemma}
For the volume case~$w = 0$, Lemma~\ref{lem:halfway} follows
immediately from the well-known Brunn-Minkowski
theorem~\cite[Theorem~6.1.1]{schneider2014convex}.  The general form
follows from the generalized Brunn-Minkowski Theorem for mixed
volumes.  We need some notions from the theory of mixed volumes, see
Busemann~\cite[Chapter~2]{busemann1958convex} for an introduction.

Minkowski has shown that for~$r$ convex shapes~$K_1, \dots, K_r
\subset \Reals^{d}$, the
volume of their linear combinations is a homogenous polynomial of
degree~$d$:
\[
\vol{\lambda_1 K_1 + \dots + \lambda_r K_r} =
\sum_{i_1=1}^r \sum_{i_2=1}^r \dots \sum_{i_d=1}^r V(K_{i_1}, K_{i_2}, \dots, K_{i_d})
\lambda_{i_1} \lambda_{i_2}\dots\lambda_{i_d},
\]
The coefficients~$V(K_{i_1},\dots,K_{i_d})$ are called \emph{mixed
  volumes}.  Setting~$r=1$ we see that~$V(K, \dots, K) = \vol{K}$.
The mixed volumes for~$r = 2$ and~$K_2 = B$, where~$B$ is the unit
ball in~$\Reals^{d}$, are known as \emph{quermassintegrals}, and
denoted
\begin{align*}
  W_0(K) & = V(K,\dots, K), \\
  W_1(K) & = V(K,\dots, K, B), \\
  W_m(K) & = V(K[d-m], B[m]), \qquad \text{for $m \in \{0, 1, \dots, d\}$}
\end{align*}
where the $P[m]$ notation means that argument~$P$ is repeated~$m$
times.

The \emph{generalized Brunn-Minkowski theorem} states that for $m \in
\{2, 3, \dots, d\}$ and convex shapes $K_0, K_1, C_1, \dots, C_{d-m}
\subset \Reals^d$, the function
\[
f(\lambda) := \Big(V(K_\lambda[m], C_1, C_2, \dots, C_{d-m}) \Big)^{1/m}
\]
is a concave function on the interval~$[0, 1]$, where 
$K_\lambda=(1-\lambda) K_0 + \lambda K_1$. (The Brunn-Minkowski
theorem is the special case~$m = d$.)
Busemann~\cite[pg.~49--50]{busemann1958convex} gives a short proof
using the Aleksandrov-Fenchel inequality.  When $C_1, \dots, C_{d-m}$
are sufficiently smooth, then the function~$\lambda \mapsto
f(\lambda)$ is a linear function only when~$K_0$ and~$K_1$ are
homothets~\cite[Theorems~6.4.4 and~6.6.9]{schneider2014convex}.  This
applies in particular to the quermassintegral case, and we obtain the
following: For $m \in \{2, 3, \dots d\}$ and convex shapes~$K_0$
and~$K_1$, the function
\[
f(\lambda) := \Big(W_{d-m}(K_\lambda) \Big)^{1/m}
\]
is concave on the interval~$[0, 1]$, and it is linear only when~$K_0$
and~$K_1$ are homothets. In particular, if $W_{d-m}(K_0)=W_{d-m}(K_1)$,
then $f(\lambda)$ is concave on the interval~$[0, 1]$, and it is constant
only when~$K_0$ and~$K_1$ are homothets.

Lemma~\ref{lem:halfway} follows from this by observing that~$\psi_0(K)
= \vol{K} = W_0(K)$ and $\psi_1(K) = \surf{K} = d
W_1(K)$~\cite[pg.~210]{schneider2014convex}.

\section{The main theorem when the optimum is unique}

Our proof strategy is to consider an optimal shape~$P$ and argue that
if~$P$ is not a spherical ball, then there is a shape~$P'$ with larger
volume or surface area.   For this argument to go through, it is
therefore necessary to first argue that an optimal shape does indeed
exist.

\begin{lemma}
  \label{lem:exist}
  Let~$Q$ be a given convex shape in~$\Reals^{d}$, for $d \geq 2$, and
  let~$w \in \{0, 1\}$.  Then there exists~$R \in \KK(Q)$ such that
  for any~$P \in \KK(Q)$ we have $\psi_w(P) \leq \psi_w(R)$.
\end{lemma}
\begin{proof}
  Let $\omega = \sup_{P \in \KK(Q)} \psi_w(P)$. For any~$i > 0$, we
  can choose $K_i \in \KK(Q)$ with $\psi_w(K_i) > \omega - 1/i$ and
  such that the origin lies in~$K_i$.  This implies that all~$K_i$ are
  contained in a ball centered at the origin whose radius is the
  diameter of~$Q$.

  By Blaschke's selection theorem~\cite{kelly1979geometry}, there is a
  subsequence~$(K_{i_j})$ that converges in the Hausdorff-sense to
  some compact convex shape~$K$.  For simplicity of presentation, we
  let $(K_i)$ denote this converging subsequence.

  By continuity of~$\psi_w$, we have~$\psi_w(K) = \omega$.  To prove
  the lemma, it now suffices to prove that~$K \in \KK(Q)$, that is,
  that~$K$ can be placed inside~$Q$ in any given orientation~$\rho$.

  We fix some rotation~$\rho$.  Since~$K_i \in \KK(Q)$, there is a
  vector~$t_i \in \Reals^{d}$ such that~$\rho K_i + t_i \subset Q$.
  Since the origin lies in~$K_i$, we have $t_i \in Q$.  Since~$Q$ is
  compact, this implies that the sequence~$(t_i)$ contains a
  subsequence converging to some vector~$t \in Q$.  Let $(t_i)$ again
  denote this subsequence, so that we have
  \begin{itemize}
  \item $\lim t_i = t \in Q$;
  \item $K_i$ converges to~$K$ in the Hausdorff-sense.
  \end{itemize}
  Let $a_i$ be the Hausdorff-distance of~$K_i$ and~$K$, and let~$b_i =
  |t_i - t|$.  It follows that the Hausdorff-distance of $\rho K_i +
  t_i$ and~$\rho K + t$ is at most~$a_i + b_i$, which implies
  that~$\rho K_i + t_i$ converges in the Hausdorff-sense to~$\rho K +
  t$.  Since $\rho K_i + t_i \subseteq Q$ and $Q$ is compact, this
  implies that $\rho K + t \subseteq Q$, so $K$ can be placed in~$Q$
  in orientation~$\rho$.
\end{proof}

We can now prove the main theorem:
\begin{theorem}
  \label{thm:main-volume}
  Let~$Q$ be a given convex shape in~$\Reals^{d}$, for $d \geq 2$,
  let~$D$ be the largest spherical ball inscribed to~$Q$, and let~$P
  \neq D$ be a convex shape that can be placed in~$Q$ in every
  orientation.  Then $\vol{P} < \vol{D}$.  If $d \geq 3$, then we also
  have $\surf{P} < \surf{D}$.
\end{theorem}
\begin{proof}
  Let~$w \in \{0, 1\}$.  By Lemma~\ref{lem:exist}, there exists $P \in
  \KK(Q)$ that maximizes~$\psi_w(P)$.  If~$P$ is not a ball, then
  there is a rotation~$\rho$ such that $P$ and~$\rho P$ are not
  homothets.  But then Lemma~\ref{lem:halfway} implies that
  $\psi_w(\frac 12 (P + \rho P)) > \psi_w(P)$.  On the other hand,
  since~$P, \rho P \in\KK(Q)$, Lemma~\ref{lem:freely} implies that
  $\frac 12(P + \rho P) \in \KK(Q)$, a contradiction to the assumption
  that~$P$ maximized~$\psi_w(P)$.
\end{proof}

%\complain{The argument actually works also for the other
%  quermassintegrals, as long as $d \geq 2 + w$.  Should we state it
%  more generally?   Or just remark this in the conclusions?}

\section{Largest perimeter}

It remains to discuss the case of maximizing the perimeter in the
plane.  It is well known (and follows for instance from the
Cauchy-Crofton formula) that $\peri{P+Q} = \peri{P} + \peri{Q}$
for any planar convex shapes~$P$ and~$Q$.

We fix an even integer~$\mu$, and define~$\rho$ to be the rotation around
the origin by angle~$\frac{2\pi}{\mu}$.  For a planar convex
shape~$P$, we define the \emph{$\mu$-average}~$\Ps$ of~$P$ to be the
set
\[
\Ps := \frac{1}{\mu} \sum_{k=0}^{\mu-1} \rho^{k} P.
\]

\begin{theorem}
  \label{thm:minkowski-average}
  Let~$P$ and~$Q$ be planar convex shapes such that $\rho^{k} P$ can
  be translated into~$Q$ for every~$k \in \{0, 1, \dots, \mu-1\}$.
  Then $\Ps$ can be translated into~$Q$, and we have $\peri{\Ps} =
  \peri{P}$.
\end{theorem}
\begin{proof}
  For $i \in \{1, 2, \dots, \mu\}$ we define
  \[
  P_{i} := \frac{1}{i} \sum_{k = 0}^{i - 1} \rho^{k} P.
  \]
  We will prove by induction that $P_{i}$ can be translated into~$Q$
  and that $\peri{P_{i}} = \peri{P}$. Since~$\Ps = P_{\mu}$, this
  implies the theorem.

  The base case is $i = 1$, where $P_{i} = P_{1} = P$.  Assume now
  that $i \in \{2, \dots \mu\}$ and that the statement holds
  for~$P_{i-1}$.  We observe that
  \[
  P_{i} = \frac{1}{i} \rho^{i} P + \frac{i - 1}{i}
  P_{i-1} = (1-\lambda) \rho^{i} P + \lambda P_{i-1} \qquad
  \text{with} \quad \lambda = 1 - \frac{1}{i}.
  \]
  Since $\rho^{i} P$ and~$P_{i-1}$ can be translated into~$Q$,
  Lemma~\ref{lem:fits} implies that $P_{i}$ can be translated
  into~$Q$. We have $\peri{P_{i}} = (1-\lambda) \peri{\rho^{i} P} +
  \lambda \peri{P_{i-1}} = \peri{P}$.
\end{proof}

Let now~$\Km$ denote a regular convex $\mu$-gon.  We denote the edges
of~$\Km$ in counter-clockwise order as $e_{1}, e_{2}, \dots, e_{\mu}$.
A convex polygon~$P$ is called a \emph{$\mu$-polygon} if every edge
of~$P$ is parallel to an edge of~$\Km$.  We represent a
$\mu$-polygon~$P$ as a vector~$\phi(P) = (a_{1}, a_{2}, \dots,
a_{\mu})$ in~$\Reals^{\mu}$, where $a_{i}$ is the length of the edge
of~$P$ with the same outward normal as~$e_{i}$.  The vector~$\phi(P)$
represents~$P$ uniquely up to translations.

We observe the following:
\begin{itemize}
\item $\phi(\alpha P) = \alpha \phi(P)$ for $\alpha > 0$;
\item $\phi(P + R) = \phi(P) + \phi(R)$;
\item If $\phi(P) = (a_1, \dots, a_\mu)$ and $\rho$ is a rotation with
  rotation angle~$\frac{2\pi}{\mu}$, then $\phi(\rho P) = (a_\mu, a_1,
  a_2, \dots, a_{\mu-1})$.
\end{itemize}

As a consequence, the $\mu$-average~$\Ps$ of~$P$ has vector
\[
\phi(\Ps) = (r, r, r, \dots, r) \qquad \text{where} \quad
r = \frac{1}{\mu} \sum_{k=0}^{\mu} a_{k} = \frac{\peri{P}}{\mu}.
\]
In other words, $\Ps$~is a regular convex~$\mu$-gon.

\begin{theorem}
  \label{thm:main-perimeter}
  Let~$Q$ be a given planar convex shape, let~$D$ be the largest disk
  inscribed to~$Q$, and let~$P \in \KK(Q)$.  Then $\peri{P} \leq
  \peri{D}$.
\end{theorem}
\begin{proof}
  Assume the claim is false, so that we have~$P \in \KK(Q)$ with
  $\peri{P} > \peri{D}$.  Let $\eps := \peri{P} - \peri{D} > 0$.  If
  we choose $\mu$ large enough, then there is a $\mu$-polygon~$R
  \subset P$ with $\peri{P} - \peri{R} < \eps/2$.  We consider its
  $\mu$-average~$\Rs$. By Theorem~\ref{thm:minkowski-average} we can
  translate $\Rs$ into~$Q$, and have $\peri{\Rs} = \peri{R} \geq
  \peri{P} - \eps/2$.  Since~$\Rs$ is a regular convex~$\mu$-gon, we
  can ensure---by making~$\mu$ large enough---that there is a disk~$D'
  \subset \Rs$ with $\peri{\Rs} - \peri{D'} < \eps/2$.  This implies
  that $\peri{P} - \peri{D'} < \eps$. Since $D' \subset \Rs \subset
  P$, the disk~$D'$ cannot be larger than~$D$ and so we have $\peri{P}
  - \peri{D} < \eps$, a contradiction to the choice of~$\eps$.
\end{proof}

%\bibliography{refs}

\end{document}